\documentclass[11pt]{article}
	
    \makeatletter
    \renewcommand\section{\@startsection {section}{1}{\z@}%
                                       {-3.5ex \@plus -1ex \@minus -.2ex}%
                                       {2.3ex \@plus.2ex}%
                                       {\normalfont\fontfamily{phv}\fontsize{16}{19}\bfseries}}
    \renewcommand\subsection{\@startsection{subsection}{2}{\z@}%
                                         {-3.25ex\@plus -1ex \@minus -.2ex}%
                                         {1.5ex \@plus .2ex}%
                                         {\normalfont\fontfamily{phv}\fontsize{14}{17}\bfseries}}
    \renewcommand\subsubsection{\@startsection{subsubsection}{3}{\z@}%
                                        {-3.25ex\@plus -1ex \@minus -.2ex}%
                                         {1.5ex \@plus .2ex}%
                                         {\normalfont\normalsize\fontfamily{phv}\fontsize{14}{17}\selectfont}}
    \makeatother
    
    \usepackage{latexsym}
\usepackage{graphicx}
\usepackage[T1]{fontenc}

\usepackage{amsmath}
\usepackage{amsfonts}
\usepackage{amssymb}
\usepackage{amsbsy}
\usepackage{amsthm}
\usepackage{algorithm}
\usepackage{algorithmic}
\usepackage[svgnames]{xcolor}
\usepackage{subfig}

\usepackage[pdftex,colorlinks=true,urlcolor=blue,citecolor=black,anchorcolor=black,linkcolor=black]{hyperref}

\usepackage{bm}
\newcommand{\BFx}{\bm{x}}
\newcommand{\BFX}{\bm{X}}
\newcommand{\BFy}{\bm{y}}
\newcommand{\BFs}{\bm{s}}
\newcommand{\BFS}{\bm{S}}
\newcommand{\BFsn}{\BFs^{\mathrm{n}}}
\newcommand{\BFsnT}{(\BFs_k^{\mathrm{n}})^\top}
\newcommand{\BFst}{\BFs^{\mathrm{t}}}
\newcommand{\BFstT}{(\BFs^{\mathrm{t}})^\top}
\newcommand{\BFSt}{\BFS^{\mathrm{t}}}
\newcommand{\BFstk}{\BFS_k^{\mathrm{t}}}
\newcommand{\BFstN}{\BFs^{\mathrm{tN}}}
\newcommand{\BFStN}{\BFS^{\mathrm{tN}}}

\newcommand{\BFsnN}{\BFs^{\mathrm{nN}}}
\newcommand{\BFsnR}{\BFs^{\mathrm{nR}}}
\newcommand{\BFc}{\bm{c}}

\newcommand{\BFG}{\bm{G}}
\newcommand{\BFGN}{\bm{G}^\mathrm{N}}
\newcommand{\BFe}{\bm{e}}

\newcommand{\sfB}{\mathsf{B}}

\newcommand{\bfmcX}{\bm{\mathcal{X}}}
\newcommand{\bfB}{\bm{B}}
\newcommand{\bfbeta}{\bm{\beta}}
\newcommand{\sfM}{\mathsf{M}}

\newcommand{\Mt}{M_k^{\mathrm{t}}}
\newcommand{\Mtz}{M_k^{\mathrm{t}0}}
\newcommand{\Mtt}{M_k^{\mathrm{tt}}}

\newcommand{\mn}{m_k^{\mathrm{n}}}
\newcommand{\mnz}{m_k^{\mathrm{n}0}}
\newcommand{\mnn}{m_k^{\mathrm{nn}}}
\newcommand{\Ms}{M_k^\mathrm{s}}
\newcommand{\Mz}{M_k^0}
\newcommand{\Qs}{Q_k^\mathrm{s}}
\newcommand{\Qz}{Q_k^0}
\newcommand{\Qn}{Q_k^\mathrm{n}}
\newcommand{\qs}{q_k^\mathrm{s}}

\newcommand{\Phis}{\Phi_k^\mathrm{s}}
\newcommand{\Phiz}{\Phi_k^0}

\newcommand{\Psikp}{\Psi_{k+1}(\varsigma_{k+1})}
\newcommand{\Psik}{\Psi_{k}(\varsigma_{k})}

\newcommand{\BFcs}{\BFc_k^\mathrm{s}}

\newcommand{\at}{a^{\mathrm{t}}}
\newcommand{\an}{a^{\mathrm{n}}}

\newcommand{\Fbar}{\bar{F}}
\newcommand{\sfH}{\mathsf{H}}
\newcommand{\sfA}{\mathsf{A}}
\newcommand{\sfAN}{\mathsf{A}^\mathrm{N}}
\newcommand{\sfAR}{\mathsf{A}^\mathrm{R}}
\newcommand{\sfHN}{\mathsf{H}^\mathrm{N}}
\newcommand{\Fbars}{\Fbar_k^\mathrm{s}}

\newcommand{\knormal}{\kappa_\mathrm{nsd}}
\newcommand{\knormc}{\kappa_\mathrm{ncr}}
\newcommand{\ktanc}{\kappa_\mathrm{tcr}}
\newcommand{\klinm}{\kappa_\mathrm{flm}}
\newcommand{\kling}{\kappa_\mathrm{flg}}
\newcommand{\khess}{\kappa_{\mathrm{u}\sfH}}

\newcommand{\kerr}{\kappa_\mathrm{e}}

\newcommand{\lipc}{\kappa_{Lc}}
\newcommand{\lipg}{\kappa_{Lg}}

\newcommand{\amax}{\kappa_{\sfA\mathrm{max}}}
\newcommand{\amin}{\kappa_{\sfA\mathrm{min}}}
\newcommand{\fmin}{f_\mathrm{min}}
\newcommand{\sigmax}{\varsigma_\mathrm{max}}
\newcommand{\sigmin}{\varsigma_\mathrm{min}}
\newcommand{\kas}{\kappa_\mathrm{d}}

\newcommand{\keeb}{b_f}
\newcommand{\kerror}{\kappa_\mathrm{err}}

\newcommand{\lagmax}{\Lambda_\mathrm{max}}
\newcommand{\ksn}{\kappa_\mathrm{bon}}
\newcommand{\kmt}{\kappa_\mathrm{mtb}}
\newcommand{\kmtn}{\kappa_\mathrm{mtn}}

\newcommand{\kpsirc}{\kappa_{\mathrm{r}1}}
\newcommand{\kpsirg}{\kappa_{\mathrm{r}2}}
\newcommand{\kdelta}{\kappa_{\Delta}}
\newcommand{\kuc}{\kappa_\mathrm{uc}}
\newcommand{\kug}{\kappa_\mathrm{ug}}
\newcommand{\kphi}{\kappa_{\phi}}
\newcommand{\kmg}{\kappa_\mathrm{mge}}
\newcommand{\kb}{\kappa_{\mathrm{u}\sfB}}

\newcommand{\trdec}{\gamma_\mathrm{dec}}
\newcommand{\trinc}{\gamma_\mathrm{inc}}
\newcommand{\sigbmax}{\varsigma_\mathrm{max}^B}
\newcommand{\ssinf}{\lambda}
\newcommand{\minsd}{\sigma_\mathrm{min}}
\newcommand{\sigmac}{\varsigma^\mathrm{C}}
\newcommand{\sigmab}{\varsigma^\mathrm{B}}
\newcommand{\sigmaxb}{\varsigma^\mathrm{B}_\mathrm{max}}
\newcommand{\Deltamax}{\Delta_\mathrm{max}}

\newcommand{\mbR}{\mathbb{R}}
\newcommand{\mbE}{\mathbb{E}}

\newcommand{\mbP}{\mathbb{P}}

\newcommand{\mcS}{\mathcal{S}}

	\usepackage{amsmath}
	\usepackage{graphicx}
	\usepackage{enumerate}
	\usepackage[numbers]{natbib}
	\usepackage{url} 
    \usepackage{xcolor}
	
    \usepackage{bm}
    \usepackage{multicol}
    \usepackage{lipsum}
    \usepackage{mwe}
    \usepackage{graphicx}
    \usepackage{subfig}
    \usepackage{amssymb}
    \usepackage{amsthm}
    \usepackage{amsfonts}
    \usepackage{epstopdf}
    \usepackage{algorithm}
\usepackage{booktabs}
\usepackage{enumitem}
\usepackage{color,soul}
\usepackage{multirow}
\usepackage{epstopdf}
\usepackage{pifont}

    \newtheorem{theorem}{Theorem}[section]
    \newtheorem{lemma}[theorem]{Lemma}
    \newtheorem{corollary}[theorem]{Corollary}
    
    \newtheorem{assumption}{Assumption}

\begin{document}
		
			\title{\bf Adaptive Sampling Trust Region Optimization for Derivative-free Stochastic Functions and Deterministic Equality Constraints}
			\author{Nicole Felice, Sara Shashaani\footnote{North Carolina State University, Raleigh, NC 27695, USA} , and Lindon Roberts\footnote{University of Melbourne, Melbourne, Australia } }
			\date{}
			\maketitle
			\bigskip
		
	\begin{abstract}
We study optimization problems with noisy zeroth-order objective observations and deterministic nonlinear equality constraints with available derivatives. We propose a constrained variant of the adaptive-sampling trust-region derivative-free optimization algorithm---ASTRO-DF. The method builds quadratic local models from estimated objective values at interpolation points within a moving trust region and promotes feasibility through a Byrd--Omojokun composite-step  based on linearized constraints, following an SQP-like framework. We prove almost sure convergence using a new constrained criticality test and present  numerical results on an equality-constrained stochastic activity network problem.

	\end{abstract}
			
	\noindent%


\section{INTRODUCTION}

\label{sec:intro}
Consider the optimization problem
\begin{align}
\min_{\BFx} \quad & \left\{f(\BFx):=\mbE_\xi[F(\BFx,\xi)]\right\} ,\quad \text{subject to}\quad
 \BFc(\BFx) = 0, \label{eq:stoch opt prob}
\end{align}
where $f:\mbR^d\to\mbR$  and $\BFc:\mbR^d\to\mbR^p$ are nonconvex. We assume that $\BFc$ and the Jacobian of $\BFc$, $\sfA:\mbR^d\to\mbR^{p\times d }$ are deterministic and available. Suppose we only have access to $F:\mathbb{R}^d\times\Xi\to\mathbb{R}$, a function whose input is a decision vector $\BFx\in \mathbb{R}^d$ and a random vector $\xi$ defined in the probability space $(\Xi, \mathcal{F}, \mathbb{P})$ and taking values in $\Xi \subseteq \mathbb{R}^q$. We do not have access to the derivative of $F$ directly, making this a derivative-free optimization (DFO) problem. We will assume LICQ everywhere, so $\operatorname{rank}(\sfA(\BFx))=p$ (i.e.~$p\leq d$), and hence, with $\sfAN(\BFx)$ forming a basis for the null space of $\sfA(\BFx)$, the first-order optimality conditions for \eqref{eq:stoch opt prob} are $\BFc(\BFx^*)=0$ and $\sfAN(\BFx^*)^\top \nabla f(\BFx^*)=0$.


Throughout the paper, we have the following standing assumptions on the function, constraints, and stochastic noise (we will not state these assumptions for every result since they must hold everywhere).
\begin{assumption} \label{as: lip and bounded}
    The objective $f$ 
    is continuously differentiable with $\lipg$-Lipchitz continuous gradients, and $f(\BFx) \geq \fmin$ for all $\BFx \in \mbR^d$. The constraint Lipchitz constant $ \lipc > 0$ satisfies, for any step $\BFs\in\mbR^d$, 
\begin{align}
    \|\BFc(\BFx+\BFs) - \BFc(\BFx) - \sfA(\BFx)\BFs\| &\leq \frac{1}{2} \lipc \|\BFs\|^2 \label{eq: lip const bound}.
\end{align}
Furthermore, 
there exist constants $\kuc,\kug$ such that $\|\nabla f(\BFx)\|\le \kug$ and $\|\BFc(\BFx)\|\le \kuc$ for all $\BFx \in \mbR^d$.
\end{assumption}
\begin{assumption}\label{as: estimation error}
Given $\kerror>0$, for all $\BFx\in\mbR^d$ and $n$ iid realizations $\{\xi_j\}_{j=1}^n$, constants $\sigma^2_f, \keeb >0$  satisfy 
    \begin{align*}
        \mbP \left ( \sum_{j=1}^n (F(\BFx,\xi_j) - f(\BFx)) \geq n \kerror \right) \leq \exp \left ( -\frac{n \kerror}{2 (\keeb \kerror + \sigma^2_f)}\right ).
    \end{align*}
\end{assumption} This assumption requires a sub-exponential tail bound on the sample average approximation of $f(\BFx)$, allowing for unbounded noise to have heavier tails than a Gaussian noise and making the framework applicable for a wide range of problems~\citep{Ha2025}. 

An area of particular interest for DFO is black-box simulation-optimization, wherein 
the only objective information available are output values returned from simulating a specific set of input designs. One such optimization problem is the Stochastic Activity Network (SAN)~\citep{avramidis1991simulation} whose objective is to select the mean duration, $\BFx$, of a set of $d$ tasks where actual task length is exponentially distributed with the given mean, to minimize the expected length of the longest path, $T$, from task $1$ to task $d$: $\min_{\BFx \in \mbR^{d}_+} \mathbb{E}[T(\BFx,\xi)],\text{ s. t. } 
\sum_{j=1}^{d} \frac{e_{j}}{x_j}\ = b.$
The problem has one equality constraint $(p=1)$; the decisions are constrained by the total cost $b$ allowed to be spent reducing mean duration of each task with associated reduction cost $e_j$ for $j= 1, \dots,d$. For such a problem, our goal is to adapt ASTRO-DF, which is an unconstrained stochastic trust-region algorithm with almost-sure convergence~\citep{Shashaani2018}, to handle deterministic nonlinear equality constraints through the integration of Sequential Quadratic Programming (SQP) techniques. We introduce SQP-ASTRO-DF and show that it achieves both first-order optimality conditions almost surely and provide numerical validation of the algorithm on the SAN problem. 




\section{Related Works}
Adaptive sampling methods for unconstrained problems were introduced in \citep{Shashaani2018}, which proved almost-sure convergence of the iterates to a stationary point.
This was extended to a worst-case complexity analysis in \citep{Ha2025}, including when common random numbers are available in the sampling regime.
These methods extend model-based derivative-free optimization (DFO) \citep{Conn2009,Larson2019,Roberts2025} methods to the stochastic regime.
Alternative adaptive sampling approaches for stochastic model-based DFO were introduced in \citep{Blanchet2019,Chen2018,Cao2024}.
The approach from \citep{Shashaani2018} differs from these in that it samples the objective estimator only until the estimated accuracy (sample standard error) is sufficiently small, rather than relying on a predetermined (if adaptive based only on the current trust-region radius) number of samples to achieve probabilistically accurate estimates.

Our SQP approach here is based on the Byrd--Omojokun step decomposition \citep{Omojokun1989}, which has been successfully used for deterministic DFO problems in the COBYQA software \citep{Ragonneau2022}.
Other model-based SQP implementations include \citep{Troltzsch2016,Hannanu2024}.
None of these works include convergence theory; our theoretical results extend the analysis of derivative-based trust-region SQP methods in \citep{Conn2000,Byrd2000} to the derivative-free case under unbounded stochastic noise.
A separate line of work extends SQP methods to problems with uniformly bounded noise (that cannot be removed via sample averaging) in \citep{Oztoprak2023,Sun2024,Oztoprak2026}.

A related line of work is on stochastic SQP methods, which are primarily targeted at machine learning settings where stochastic gradient estimates for the objective are available (and exact function values and gradients are usually assumed for all constraints).
This includes both linesearch \citep{Berahas2021,Curtis2024,ONeill2026} and trust-region methods \citep{Fang2024}, and adaptive methods, where random quantities are assumed to be accurate with some guaranteed probability \citep{Fang2024storm,Fang2026complexity,Fang2026}.
These latter approaches extend the approach of \citep{Chen2018} and related works above to the constrained case (and can in principle be applied to the DFO context, although that is not their main focus), whereas here we handle randomness in objective evaluations using the analysis approach of \citep{Shashaani2018}.

\section{Inexact Function and Gradient Evaluation Using Trust Regions}
Our choice of trust-region methods here is due to their ability to synchronize the estimation and gradient approximation error with the stationarity through an important adaptive algorithm parameter, the trust-region radius. At each iteration $k$, the trust-region radius $\Delta_k$ defines our interpolation set (perturbation used for finite-difference approximation of the gradient), controls the allowable search step's size for the next iterate (which in turn controls the search direction as well), and tracks the stationarity measure at the current iterate that helps with determining the estimation accuracy.

Suppose we are in the deterministic setting and the algorithm generates a sequence of iterates $\{\BFx_k\}$. At each iteration, given a base point $\BFx_k$ and $d'$ interpolation points $\{\BFx_k^i,\ i=1, \dots, d'\}$, if we had access to the true objective function value $f_k := f(\BFx_k)$, we would build a \emph{fully-linear} interpolation model (as will be defined in Theorem~\ref{th: fully linear}) for the objective 
\begin{align}
    f(\BFx) \approx q_k(\BFx) := f_k + g^\top_k(\BFx-\BFx_k) + \frac{1}{2} (\BFx-\BFx_k)^\top \mathsf{B}_k(\BFx-\BFx_k), \label{eq: q}
\end{align}
where $g_k := \nabla q_k(\BFx_k) \in \mbR^d$ is an approximation of $\nabla f_k$ and $\mathsf{B}_k := \widehat{\nabla^2 l} (\BFx_k, \BFy_k) \in \mbR^{d \times d}$ approximates the Hessian of the Lagrangian 
    $$\nabla ^2l (\BFx_k, \BFy_k) : = \nabla ^2f(\BFx_k) - \sum_{i =1}^p [\BFy_{k}]_{i} \nabla ^2c_i(\BFx_k),$$ 
for a set of chosen Lagrange multipliers $\BFy_k \in \mbR^p$. We leave the method of approximating the Hessian unspecified, as any method resulting in a uniformly bounded $\sfB_k$ is sufficient. For a given trust-region radius $\Delta_k$, we make the following assumptions on the accuracy of our constructed model, which are standard in DFO.
\begin{assumption}\label{as: bounded true Hess}
    $\sfB_k$ is uniformly bounded, meaning there exists $\kb > 0 $ such that $\|\sfB_k\| \leq \kb$ for all $k$.
\end{assumption}

\begin{assumption}\label{as: bounded lag poly}
   Let $\{\ell_k^i\}$ be the Lagrange polynomials associated with the $i$-th interpolation point for $i = 0,1,\dots, d'$. We assume that for all $k$, the interpolation set is well-poised over $\Delta_k$, meaning there exists a constant $\lagmax >0$ such that 
   \begin{align*}
        \sup_{\BFx \in B(\BFX_k, \Delta_k)} |\ell^i_k(\BFx)| \leq \lagmax \quad \text{for} \:\: i = 0, 1, \dots, d' .
    \end{align*}
\end{assumption}
\begin{theorem}{\citep{Conn2009}} \label{th: fully linear}
    Suppose Assumptions~\ref{as: lip and bounded}, \ref{as: bounded true Hess}, \ref{as: bounded lag poly} hold. Then, the model $q_k$ \eqref{eq: q} is a fully linear approximation for $f$, meaning there exist $\klinm,\kling>0$ such that $\forall \BFx \in \{\BFX_k+\BFs:\ \BFs\in\mbR^d, \|\BFs\|\le a \Delta_k\}$ 
    \begin{align}
        | f(\BFx) - q_k(\BFx) |  \leq \klinm a^2 \Delta_k^2 \label{eq: fully linear model}\quad \text{and} \quad
        \| \nabla f(\BFx) - \nabla q_k(\BFx) \|  \leq \kling a \Delta_k. 
    \end{align}
\end{theorem}

In the stochastic setting, the sequence of iterates becomes stochastic, which we denote by $\{\BFX_k\}$. We use sample-average approximation to estimate $f$ at the iterate and interpolation set 
as well as at a candidate solution $\BFX_k^s$ recommended by the trust-region subproblem. The estimator function is denoted $$\Fbar(\BFx,n)=\frac{1}{n}\sum_{j=1}^{n}F(\BFx,\xi_j)$$ using iid $\xi_j$'s. Instead of a fixed sample size $n$, we use an adaptive one, $N_k^i$, to be just large enough to ensure an estimate of required accuracy. 
To approximate the first and second order information, we perform interpolation over the interpolation set $\bfmcX_k=\{\BFX_k^i, i=0,1,\cdots,d'\}$ of $d'$ design points that satisfies Assumption~\ref{as: bounded lag poly}. We determine the necessary number of replications at each point using the following adaptive sampling rule and constants $\minsd, \kas >0$
\begin{align}\label{eq: adapt N}
   N_k^i := N\left ( \BFX^i_k\right )=\min \left \{n:\:\frac{\max\left\{ \minsd, \hat \sigma (\BFX_k^i
   ,n) \right \}}{\sqrt{n}} \leq \kas \frac{\Delta_k^2}{\sqrt{\ssinf_k}}\right \}\quad i\in\{0,1,\ldots,d',s\}
\end{align}
where  $$\hat \sigma(\BFx,n) = ((n-1)^{-1} \sum_{j=1}^n \left ( F(\BFx, \xi_j) - \Fbar(\BFx,n)\right)^2)^{1/2},$$ 
and  $\ssinf_k = \mathcal{O}(\log^{1+\epsilon}k)$, for some small $\epsilon>0$, is a deterministic sequence that deflates the accuracy requirement in the right-hand side of \eqref{eq: adapt N}. Here $\minsd$ is a user-defined parameter avoiding an early stopping for the sample size due to a probable 
largely underestimated $\hat \sigma(\BFx,n)$.

Given a base point $\BFX_k$, we estimate $q_k$ as defined in \eqref{eq: q} by $Q_k$ as will be described next. Following the interpolation strategy of  \citep{Ha04052025}, $d'=2d$ interpolation points are selected along coordinate basis within $\Delta_k$ of the center point $\BFX_k^0: = \BFX_k$ of:
\begin{align*}
    \bfmcX_k:=\left \{\BFX_k^0, \BFX_k^1 = \BFX_k^0 +\BFe^1 \Delta_k, \dots, \BFX_k^d = \BFX_k^0 + \BFe^d \Delta_k, \BFX^{d+1} = \BFX_k^0 - \BFe^1 \Delta, \dots, \BFX_k^{2d} = \BFX_k^0 - \BFe^d \Delta_k \right \},
\end{align*} 
where $ \{ \BFe^1, \dots, \BFe^d\}$ forms the $\mbR^d$ standard basis. Let $$\bfB(\BFx) := (1, [\BFx]_1, [\BFx]_2, \dots,[\BFx]_d, \frac{1}{2}[\BFx]^2_1, \frac{1}{2}[\BFx]_2^2, \dots, \frac{1}{2}[\BFx]_d^2)$$ be a polynomial basis on $\mbR^d$. We then solve $$\sfM_k(\bfB, \bfmcX_k) \bfbeta= [\Fbar(\BFX_k,N^0_k); (\Fbar(\BFX_k^i, N^i_k)) - \Fbar(\BFX_k, N^0_k))_{i=1}^{2d}]^\top$$ for  $\bfbeta = ( \beta_0, \beta_1, \dots, \beta_{2d})$, where $\sfM_k$ is an invertible Vandermonde matrix, 
to build a fully-linear local model that approximates $f$ in $\Delta_k$-neighborhood of 
$\BFX_k$: 
\begin{align}
 Q_k(\BFX_k+\BFs) := \Fbar(\BFX_k,N_k) + \BFG_k^\top\BFs + \frac{1}{2}\BFs^\top \sfH_k\BFs, \quad \BFs \in \mbR^d,\ \|\BFs\|\le \Delta_k.\label{eq: Q model 2}
\end{align}
In \eqref{eq: Q model 2}, $\BFG_k : =[\beta_1,\ldots,\beta_d]$ is the approximated noisy gradient $\nabla f(\BFX_k)$ and $$\sfH_k : = \text{diag} \left ( [\beta_{d+1}, \dots, \beta_{2d}]\right ) - \sum_{i =1}^p [\BFy_{k}]_{i} \nabla ^2c_i(\BFx_k), $$ approximated by any method that results in uniform boundedness of $\sfH_k$.

To simplify notation, we will use $f_k^i:=f(\BFX_k^i),\Fbar_k^i:=\Fbar(\BFX_k^i,N_k^i)$ for $i\in\{0,1,\ldots,2d,\mathrm{s}\}$. Then, using the interpolation set $\bfmcX_k$, we can express the models in terms of the Lagrange polynomials $\ell^i_k$ for $i=0,1, \dots, 2d$ \citep{Roberts2025}:
\begin{align*}
    q_k(\BFx) = \sum_{i =0}^{2d} f_k^i \ell^i_k(\BFx), \quad \text{and} \quad Q_k(\BFx) = \sum_{i =0}^{2d} \Fbar_k^i\ell^i_k(\BFx).
\end{align*}

Following the needed assumptions to obtain Theorem~\ref{th: fully linear}, we add the following for the stochastic setting.
\begin{assumption}\label{as: bounded hess}
   There exists a constant $\khess$ such that almost surely $\|\sfH_k\|\leq \khess -1$ for all $k$. 
\end{assumption}

The next two theorems adress the error of the model $Q_k$ and its gradient at the evaluated points.
\begin{theorem}{\citep{Ha2025}}\label{th: est error bound}
If $\Fbar(\BFX_k^i, N_k^i)$ is the estimate of $f(\BFX_k^i)$ using the adaptive sampling rule \eqref{eq: adapt N}, then for any user-defined constant $\kerror >0$, $\alpha\in\{0,1,2\}$ and $i\in\{0,1,\ldots,2d,s\}$
\begin{align*}
    \mbP\left ( \left| \Fbar(\BFX_k^i, N_k^i) - f(\BFX_k^i) \right | > \kerror \Delta^\alpha_k \quad \text{i.o.}\right) = 0.
\end{align*}
\end{theorem}
\begin{theorem}{\citep{Ha2025}}\label{th: gradient error bound}
   Suppose Assumption~\ref{as: bounded hess} holds and let $Q_k$ of $f$ be a stochastic quadratic model with diagonal Hessian constructed as defined in \eqref{eq: Q model 2} on $\{\BFX_k+\BFs:\ \BFs\in\mbR^d,\ \|\BFs\|\le \Delta_k\}$. Then, given $\kerror>0$ (as in Theorem~\ref{th: est error bound}), the constant $$\kmg=\sqrt\frac{d}{2}\left(\frac{\lipg+\khess}{2}+2\kerror\right)$$ satisfies 
\begin{align}\label{eq: grad error}
    \mbP\left (\left \|  \nabla Q_k (\BFX_k) - \nabla f(\BFX_k)\right \| > \kmg \Delta_k \quad \text{i.o.} \right) = 0. 
\end{align} 
\end{theorem}
These two results guarantee with probability one that, given a user-defined $\kerror$, any realization of the random sequence $\{\BFX_k\}$ eventually has the estimation error bounded  by $\kerror\Delta_k^2$ and the model gradient error bounded by $\kmg\Delta_k$.


\section{SQP-ASTRO-DF Algorithm}

\begin{algorithm}
\caption{SQP-ASTRO-DF} \label{alg: sqp}
\textbf{Inputs:} Initial point $\BFx_0 \in \mbR^d$, initial trust-region radius $\Delta_0 >0$, initial penalty parameter $\sigmin >0$.\\
\textbf{Parameters:} $\eta \in (0,1)$, $0 < \trdec < 1 \leq \trinc$, $\an, \at \in (0,1)$, $\nu \in (0,1)$, $\tau_1 \geq 1$, $\tau_2 > 0$ (with $(\tau_1-1)\tau_2 >0$), $\mu>0$, $\sigbmax >0$, $\Deltamax > \Delta_0$. 
    \begin{algorithmic}[1]
        \FOR{$k = 0,1,2, \dots$}
            \STATE Build the fully linear model $Q_k$ using $N_k^i$ i.i.d. replications satisfying \eqref{eq: adapt N} at each interpolation point $\{\BFX^i_k\}_{i=0,1, \dots, 2d}$ to obtain $\Fbar_k^i$.
            \STATE Compute $\BFsn_k$ \eqref{eq: normal step} satisfying Assumptions \ref{as: normal enough} and \ref{as: normal cauchy}.
            \STATE Compute $\BFSt_k$ \eqref{eq: tangent step} satisfying Assumption~\ref{as: tangent cauchy}  to get $\BFS_k := \BFsn_k + \BFSt_k$, and set $\BFX_k^s=\BFX_k+\BFS_k$.
            \STATE Given $\mnz - \mnn$, $\Mtz - \Mtt$, and $\Qz - \Qn$, set $\sigmac_k$ according to \eqref{eq: sigma c}. \label{line: set sigma c}
            \STATE Set the penalty parameter
            \[
                \varsigma_k : = \begin{cases}\max\{ \sigmac_k, \tau_1\varsigma_{k-1}, \varsigma_{k-1} + \tau_2\}, \quad &\text{if} \quad \varsigma_{k-1} < \sigmac_k,\\
                \varsigma_{k-1} ,&\text{o.w.}
                \end{cases}
            \] \label{line: set sigma}
            \STATE Evaluate $\BFc_k^s$ and simulate $N_k^s$ i.i.d. replications to obtain $\Fbar_k^s$ and compute $\rho_k$  and $\pi_k$ \eqref{eq: improve ratio}.
            \IF{$\rho_k \geq \eta$ and $\pi_k \geq \mu \Delta_k$}
            \STATE Set $\BFX_{k+1} = \BFX_k^s$ and $\Delta_{k+1} = \min \{ \trinc \Delta_k, \Deltamax\}$.\label{line: deltamax}
            \ELSE 
            \STATE Set $\BFX_{k+1} = \BFX_k$ and $\Delta_{k+1} = \trdec \Delta_k$.       
            \ENDIF
        \ENDFOR
    \end{algorithmic}
\end{algorithm}

Here, we describe Algorithm~\ref{alg: sqp} and the standard SQP operations within our stachastic framework. To determine the quality of a given solution $\BFx$, we use the following 2-norm merit function
\begin{align}
    \Phi(\BFx,\varsigma) := \bar F(\BFx, N(\BFx)) + \varsigma \|\BFc(\BFx)\|, \label{eq: est merit}
\end{align}
where $\varsigma>0 \in \mbR$ is a penalty parameter that controls how strictly the constraints are prioritized. At iteration $k$, we approximate \eqref{eq: est merit} with
\begin{align}
    M_k(\BFx) = Q_k(\BFx) + \varsigma_k\|\BFc_k + \sfA_k(\BFx-\BFX_k)\|
\end{align}
using the induced merit function given $Q_k$, $\sfA_k := \sfA(\BFX_k)$ and $\BFc_k := \BFc(\BFX_k)$, and a linear approximation of the constraints $\BFc(\BFx) \approx \BFc_k+ \sfA_k(\BFx-\BFX_k)$. At each iteration, we determine the next candidate point  $\BFX_{k}^s := \BFX_k+ \BFS_k$ with step $\|\BFS_k\| \leq a\Delta_k, a \in (0,2)$ that approximately minimizes $M_k$. For ease of exposition, we use the shortened notation $\Mz:=M_k(\BFX_{k}),\Ms:=M_k(\BFX_{k}^s)$, $\Qz : = Q_k(\BFX_k)$, $\Qs:=Q_k(\BFX_{k}^s)$, $\Phiz:=\Phi(\BFX_k,\varsigma_k)$, $\Phis:=\Phi(\BFX_k^s,\varsigma_k)$, $\BFcs := \BFc(\BFX_k^s)$, and $\qs:= q_k (\BFX_{k}^s)$. In the interest of space, we provide only brief outlines for the proofs of Lemmas~\ref{lem: model error}--\ref{lem: finite Psi}.
\begin{lemma}\label{lem: model error}
    Suppose Assumptions~\ref{as: bounded lag poly} and \ref{as: bounded hess} hold. Then for any $\|\BFS_k\| \leq \Delta_k$,
    \begin{align}
        \mbP \left (|\Phis - \Ms| > \kerr (1+ \varsigma_k) \Delta^2_k \quad \text{i.o.}\right)=0,
    \end{align}
    where $\kerr := \kerror (1+ (2d+1)\lagmax)+a^2(\klinm+ \frac{1}{2} \lipc)$. 
\end{lemma}
\begin{proof}
Given Assumption~\ref{as: bounded lag poly}, by Theorem~\ref{th: est error bound} taking $\alpha=2$ for sufficiently large $k$, we can use Theorem~\ref{th: fully linear} to get $$ \|\Fbars - \Qs \| \leq (\kerror + \klinm a^2 + (2d+1)\kerror\lagmax) \Delta_k^2.$$ Then  
\[
|\Phis - \Ms| \leq (\kerror + \klinm a^2+ (2d+1)\kerror\lagmax) \Delta_k^2 + \varsigma_k \frac{\lipc}{2}\|\BFS_k\|^2,
\] is obtained by using \eqref{eq: lip const bound}.
\end{proof}

We use a Byrd--Omojokun approach \citep{Omojokun1989} to determine $\BFS_k$, decomposing it into a normal and tangent step, that is, $\BFS_k = \BFsn_k + \BFSt_k\in  \mbR^d$. The normal step (which given the iterate will be deterministic)
\begin{align}
    \BFsn_k \approx \underset{\BFsn \in \mbR^d }{\arg \min}\: \left\{\mn(\BFsn) : = \|\sfA_k\BFsn+ \BFc_k\|:\ \|\BFsn\| \leq \an\Delta_k \right\}, \label{eq: normal step}
\end{align}
is computed first to improve feasibility for some $\an \in (0,1)$. Once $\BFsn_k$ has been determined, the tangent step (which will depend on the new estimated quantities and hence stochastic denoted with a capital letter)
\begin{align}
    \BFSt_k\approx \underset{\BFst \in \mbR^d}{\arg \min}\: \left\{\Mt (\BFst) := (\BFG_k+\sfH_k\BFsn_k)^\top \BFst + \frac{1}{2}\BFstT\sfH_k\BFst :\ \sfA_k \BFst = 0,
    \: \|\BFst\| \leq \at\Delta_k\right\}, \label{eq: tangent step}
\end{align}
uses the remaining trust-region budget to improve the quality of the objective while respecting the linearized constraints for some $\at \in (0,1)$ and $a=\an+\at$.
We ensure that the main function of $\BFsn_k$ is to improve feasibility, requiring that $\BFsn_k$ be approximately orthogonal to $\BFSt_k$. Let $\sfAR_k \in \mbR^{d \times p}$ and $\sfAN_k \in \mbR^{d \times (d-p)}$ form a basis for the column space of $\sfA^\top_k$ and the null space of $\sf A_k$ respectively. Thus the normal step can be written $ \BFsn_k = \sfAR_k \BFsnR_k + \sfAN_k \BFsnN_k. $
We impose the following assumptions on our construction of $\BFsn_k$.
\begin{assumption} \label{as: ortho AN}
      At each iteration $k$, the matrix $\sfAN_k$ has orthonormal columns, i.e., $\|\sfAN_k\|=1$, and the singular values of $\sfAR_k$ match those of $\sfA_k^\top$ (via QR factorization), implying $\|\sfAR_k \BFsnR\| =  \|\sfA_k\BFsnR\|$ for any $\BFsnR \in \mbR ^{d-p}$. 
\end{assumption}
To ensure that the majority of $\BFsn_k$ is composed of $\BFsnR_k$ we have the following assumption.

\begin{assumption} \label{as: normal enough}
The normal step computed in \eqref{eq: normal step} satisfies $\|\BFsnN_k\| \leq \knormal \|\BFsnR_k\|$ for some $\knormal > 0$, and $\BFsnR_k =0$ whenever $\sfAN_k \BFc_k = 0$. This implies $\BFsn_k = 0$ whenever $\sfAN_k \BFc_k = 0$.
\end{assumption}
Let $\mnz: = \mn(0)$ and $\mnn : = \mn( \BFsn)$. We assume that the normal step computed by the algorithm approximately reduces the normal model as follows. 
\begin{assumption} \label{as: normal cauchy}
    The normal step $\BFsn_k$ computed in \eqref{eq: normal step} satisfies
    , with a constant $\knormc\in(0,1)$,
    \begin{align*}
        \mnz - \mnn \geq \knormc \frac{\|\sfA_k^\top\BFc_k\|}{\|\BFc_k\|} \min \left( \an \Delta_k, \frac{\amin\|\BFc_k\|}{1+\|\sfA_k^\top\sfA_k\|}\right).
    \end{align*} 
\end{assumption} This assumption, as shown in \citep{Conn2000}, is achievable by a Cauchy step. \eqref{eq: tangent step} can then be transformed into the following unconstrained trust-region problem with $\BFSt = \sfAN\BFStN$
\begin{align}
    \BFStN_k \approx & \underset{\BFstN \in \mbR^d}{\arg\min}\  \left\{(\BFGN_k) ^\top \BFstN + \frac{1}{2}(\BFstN)^\top \sfHN_k\BFstN:\  \|\sfAN_k\BFstN\| \leq \at\Delta_k\right\}, \label{eq: tangent step unconstrained}
\end{align}
where $$\BFGN_k := \sfAN_k(\BFG_k+ \sfH_k \BFsn_k)$$ and $$\sfHN_k := (\sfAN_k)^\top\sfH_k\sfAN_k.$$ We then set $$\BFSt_k:=\sfAN_k\BFStN_k.$$ 
Letting $\Mtz : = \Mt(0)$ and $\Mtt : = \Mt(\BFSt_k)$, we have the following assumption on the improvement in the tangent model, which again is achievable by a Cauchy step \citep{Conn2000}. 
\begin{assumption}\label{as: tangent cauchy}
    The tangent step $\BFstk$ computed in \eqref{eq: tangent step unconstrained} satisfies 
    $\|\BFstk\| \leq \at \Delta_k$ and with a $\ktanc \in (0,1)$, 
    \begin{align*}
        \Mtz - \Mtt \geq \ktanc \|\BFGN_k\| \min \left( \at \Delta_k, \frac{\|\BFGN_k\|}{\khess}\right ).
    \end{align*}
    \end{assumption}
    
At each iteration, the expansion or contraction of $\Delta_{k+1}$ as well as acceptance of $\BFX_k^s$ as the new iterate is determined using the model improvement ratio, $\rho_k$, and a new criticality measure, $\pi_k$, defined as 
\begin{align}\label{eq: improve ratio}
    \rho_k := \frac{\Phiz - \Phis}{\Mz - \Ms} \quad \text{and} \quad \pi_k := \|\BFc_k\| + \| (\sfAN_k)^\top \BFG_k\|.
\end{align}
A successful iteration deems $\Delta_{k+1} = \trinc \Delta_k$ if $\rho_k \geq \eta$ and $\pi_k \geq \mu \Delta_k$, given parameters $\eta$ and $\mu >0$ and otherwise $\Delta_{k+1} = \trdec \Delta_k$ for $0 < \trdec < 1 \leq \trinc$. This is designed so that $\Delta_k \to 0$ (which we later prove almost surely) implies $\pi_k \to 0$, meaning both $\|\BFc_k\| \to 0$ and $\|(\sfAN_k)^\top \BFG_k\| \to 0$ and the algorithm approaches feasibility and stationarity in the limits. 
The following model improvement simplifications will be used in later proofs. The improvement in the objective model after taking a normal step $\BFsn_k$ can be simplified to 
\begin{align}\label{eq: Qn}
    \Qz - \Qn := Q_k(\BFX_k) - Q_k(\BFX_k+\BFsn_k) = -\BFG_k^\top \BFsn_k - \frac{1}{2}\BFsnT \sfH_k \BFsn_k  ,
\end{align}
and similarly the improvement in the merit function model after taking a composite step $\BFS_k$ is
\begin{equation}\label{eq:model-decompose}
    \Mz-\Ms = \Mtz - \Mtt + \varsigma_k( \mnz - \mnn) + \left(\Qz -\Qn\right).
\end{equation}
If the normal and tangent step are properly chosen to improve their respective models, the first two components on the right-hand-side of \eqref{eq:model-decompose} will be positive. However, it is very possible that because $\BFsn_k$ is chosen without consideration of the objective, the improvement in the objective after taking such a step, i.e., $\left(\Qz -\Qn\right)$ may be negative. To ensure $\Mz-\Ms>0$ we enforce
\begin{align}\label{eq: allowable sigma}
    \Mz-\Ms \geq \nu \varsigma_k(\mnz-\mnn),
\end{align}
where $\nu \in (0,1)$. To satisfy \eqref{eq: allowable sigma} we compute the smallest allowable $\varsigma_k$ as
\begin{align}\label{eq: sigma c}
    \sigmac_k := \max \left (\sigmab_k , \frac{-\left[ ( \Qz - \Qn)+\left( \Mtz - \Mtt\right )\right ] }{(1- \nu) (\mnz - \mnn)} \right )
\end{align}
for some $\sigmab_k \in \left [0, \sigmaxb\right ]$ with $\sigmaxb>0$ chosen to prevent a long sequence of small incremental increases in $\sigmac_k$. Note, $\varsigma_k$ is non-decreasing so the smallest penalty value is the initial value $\sigmin$.

\section{Convergence Analysis}
We first state an additional assumption, which is again standard, to simplify the analysis.
\begin{assumption} \label{as: A bounded}
There exist $\amin>0$ and $\amax > 0$ such that $\amin \leq \sigma_i(\sfA_k)  \leq \amax$ for all $k$ and all singular values $\sigma_i$ of $\sfA_k$. 
  
\end{assumption} This is a common assumption for deterministic nonlinear constraints that also yields $$\amin \|\BFc_k\| \leq \|\sfA_k^\top \BFc_k\| \leq \amax \|\BFc_k\|$$ for all $\BFc_k$, by the classical LICQ assumption. 

Given that Assumptions~\ref{as: ortho AN}, \ref{as: normal enough}, \ref{as: normal cauchy}, and \ref{as: A bounded} hold, we have the following result from  \citep{Conn2009}.

\begin{lemma}\label{lem:norm step bound by c}
For all iterations $k$  the constant $$\ksn :=\min\left(\knormc \amin,\frac{\knormc\amin^4}{2(1+\amax^2)(\amax+\knormal)}\right)$$ satisfies
\begin{align}\label{eq: norm step bound by c}
    \ksn \|\BFsn_k\| \leq \mnz-\mnn \leq \|\BFc_k\|.
\end{align}
\end{lemma}
The next lemma shows that, with probability 1, for large enough $k$ the model reduction is at least a combination of reductions due the normal and tangent steps. 
\begin{lemma}\label{lem: model bound by tanget and normal}
   Under Assumptions \ref{as: bounded hess}, \ref{as: ortho AN}, \ref{as: normal enough}, \ref{as: normal cauchy}, and \ref{as: A bounded} the  constant $$\kmtn:=\left (\frac{\kmg \Deltamax + \kug}{\ksn}  + \frac{\khess \kuc}{2 \ksn^2}\right )$$ satisfies
\begin{align*}
   \mbP \left ( \Mz - \Ms < \Mtz - \Mtt + (\varsigma_k - \kmtn) (\mnz - \mnn)\quad \text{i.o.}\right) = 0.
\end{align*}
\end{lemma}
\begin{proof}
    If we use an adaptive sampling scheme as described in \eqref{eq: adapt N}, by Assumption~\ref{as: lip and bounded}, Theorem~\ref{th: gradient error bound}, and the triangle inequality for sufficiently large $k$, $$ \|\BFG_k \| \leq \kmg \Deltamax + \kug$$ following Line~\ref{line: deltamax} of Algorithm~\ref{alg: sqp}. Then we have 
\begin{align*}
     \Mz-\Ms \geq  \Mtz - \Mtt + \varsigma_k( \mnz - \mnn) - \kmtn
     ( \mnz - \mnn),
\end{align*} by Assumption~\ref{as: bounded hess}, Lemma~\ref{lem:norm step bound by c} and using \eqref{eq: Qn} and \eqref{eq:model-decompose}. 
\end{proof}
An important corollary follows ensuring that the merit function parameter will remain at its capped value almost surely.
\begin{corollary}\label{cor: bounded pentalty}
 Suppose Assumptions \ref{as: bounded hess}, \ref{as: ortho AN}, \ref{as: normal enough}, \ref{as: normal cauchy}, \ref{as: tangent cauchy}, and \ref{as: A bounded} hold. There exists $\sigmax>0$ such that $$\mbP(\varsigma_k \neq  \sigmax\ \text{i.o.})=0.$$
\end{corollary}
\begin{proof}
From Assumption~\ref{as: tangent cauchy}, we have $\Mtz - \Mtt \ge 0$. So for sufficiently large $k$, Lemma~\ref{lem: model bound by tanget and normal} implies $$\Mz-\Ms \geq (\varsigma_k - \kmtn) (\mnz - \mnn)$$ and \eqref{eq: allowable sigma} holds whenever $\varsigma_k \geq \frac{\kmtn}{1-\nu}$. Following Line~\ref{line: set sigma c} and \ref{line: set sigma} of Algorithm~\ref{alg: sqp}, either $\varsigma_k = \varsigma_{k-1}$ or $$\varsigma_k \geq \varsigma_{k-1} + \max \{ \tau_1, \sigmin \tau_2\},$$ where the latter only occurs if $\varsigma_{k-1} < \sigmac_k$. However, if $$\varsigma_{k-1} > \max \left( \frac{\kmtn}{1-\nu}, \sigmaxb\right ),$$ the above reasoning shows $\varsigma_{k-1} \geq \sigmac_k$ and $\varsigma_k$ remains unchanged.  
\end{proof}

In the next lemma, we prove that for large enough $k$, the model reduction will be at least a fraction of the reduction in the tangent model almost surely.
\begin{lemma} \label{lem: model improvment bound by tangent}Suppose Assumptions \ref{as: bounded hess}, \ref{as: ortho AN}, \ref{as: normal enough}, \ref{as: normal cauchy}, \ref{as: tangent cauchy}, and \ref{as: A bounded} hold. Then $$\kmt : = \frac{1}{2} \min \left ( 1, \frac{\nu \sigmax}{\kmtn- \sigmax}\right )$$ satisfies
\[\mbP(\Mz - \Ms < \kmt \left(\Mtz - \Mtt\right)\text{ i.o.})=0.\]

\end{lemma}
\begin{proof}
    Choose $k$ sufficiently large such that Lemma~\ref{lem: model bound by tanget and normal} and Corollary~\ref{cor: bounded pentalty} both hold and examine two cases: $$-\frac{1}{2} \left ( \Mtz - \Mtt\right ) \leq ( \sigmax - \kmtn) ( \mnz - \mnn)$$ and otherwise. Using Lemma~\ref{lem: model bound by tanget and normal} and letting $\kmt = \frac{1}{2}$ proves the first case. Alternatively, Assumptions~\ref{as: normal cauchy} and \ref{as: tangent cauchy} imply $$\Mtz - \Mtt\ge 0, \mnz - \mnn \geq 0,$$ which only holds if $\kmtn \geq \sigmax$. Using \eqref{eq: allowable sigma} and $\varsigma_k = \sigmax$ proves the lemma for $$\kmt =\frac{\nu \sigmax}{2(\kmtn- \sigmax)}. $$
\end{proof}

We define the \emph{scaled} merit function as $$\Psi_k(\varsigma)=\frac{\Phi_k^0(\varsigma)-f_\mathrm{min}}{\varsigma}=\frac{\Fbar_k-f_\mathrm{min}}{\varsigma}+\|\BFc_k\|.$$ Observe that if $k$ is successful, then $$\Psikp-\Psik\le -\eta \frac{\Mz-\Ms}{\varsigma_k}$$ where we use $\varsigma_k$ being non-decreasing and upper bounded by Corollary~\ref{cor: bounded pentalty}. Conversely, for unsuccessful iterations, we can write using Theorem~\ref{th: est error bound} that for large enough $k$ and a user-defined $\kerror$, almost surely 
\begin{equation}\label{eq:unsuccess psi reduction}
    \Psikp-\Psik\le \frac{\Fbar_{k+1}-\Fbar_k}{\varsigma_k}\le \frac{2\kerror\Delta_k^{\alpha}}{\sigmin}
\end{equation}
for $\alpha\in\{0,1,2\}$ and using the fact that for an unsuccessful iteration $\Delta_{k+1}=\trdec\Delta_k<\Delta_k$. 

We next show a useful model reduction result from how the merit penalty is selected in Algorithm~\ref{alg: sqp}.
\begin{lemma}\label{lem: model reduction using normal reduction}
    Suppose Assumptions~\ref{as: normal cauchy} and \ref{as: A bounded} hold. Then \[\Mz-\Ms\ge \nu\varsigma_k\knormc\amin\min\left(\an\Delta_k,\frac{\amin\|\BFc_k\|}{1+\amax^2}\right). \]
\end{lemma}
\begin{proof}
    Note that $\varsigma_k\ge \varsigma_k^c$
    along with \eqref{eq: sigma c} leads to $$\Qz-\Qn+\Mtz-\Mtt\ge (\nu-1)\varsigma_k(\mnz-\mnn).$$ Using \eqref{eq:model-decompose} 
    and Assumption~\ref{as: normal cauchy} completes the proof.
\end{proof}
The remaining ingredient is to show that the merit function remains bounded below with probability 1, which holds given the definition of $\Psik$ for any sufficiently large $k$ such that Theorem~\ref{th: est error bound} is satisfied. 
\begin{lemma}\label{lem: finite Psi}
    For any $\kerror>0$, $\mbP(\Psik<-\kerror/\sigma_{\min}\ \text{i.o.})=0$. That is,  $$\mbP(\Psik\to-\infty)=0.$$
\end{lemma}
We are now ready to prove an important result for derivative-free trust region methods, namely that the trust-region radius will vanish almost surely.
\begin{lemma}
    Suppose Assumptions~\ref{as: bounded hess}, \ref{as: ortho AN}, \ref{as: normal enough}, \ref{as: normal cauchy}, \ref{as: tangent cauchy}, and \ref{as: A bounded} hold. Then $$\mbP(\Delta_k\to 0)=1.$$ \label{lem: delta converges}
\end{lemma}
\begin{proof}
    If we have finitely many successful iterations, then this holds trivially. Suppose we have infinitely many successful iterations, the set of which we call $\mathcal{S}$. For all $k\in\mcS$, if $$\|\BFc_k\| \ge \frac{\mu}{2(1+\frac{\khess}{\ksn})}\Delta_k$$ then we  can use Lemma~\ref{lem: model reduction using normal reduction} to write 
    \[\Psikp-\Psik\le-\eta\nu\knormc\amin\min\left(\an,\frac{\amin\mu}{2(1+\amax^2)(1+\frac{\khess}{\ksn})}\right)\Delta_k:=-\kpsirc\Delta_k.\] 
    Now consider the case of $$\|\BFc_k\| < \frac{\mu}{2(1+\frac{\khess}{\ksn})}\Delta_k.$$ Since $k\in\mcS$,  $$\|\BFc_k\|+\|(\sfAN_k)^\top \BFG_k\|\ge \mu\Delta_k$$ by Algorithm~\ref{alg: sqp}. Hence, $$\|(\sfAN_k)^\top \BFG_k\|\ge \left(\mu-\frac{\mu}{2(1+\frac{\khess}{\ksn})}\right)\Delta_k.$$ 
    Next, we observe from the definition of $\BFGN_k$ and Lemma~\ref{lem:norm step bound by c} that $$\|\BFGN_k-(\sfAN_k)^\top\BFG_k\|\le \|\sfAN_k\|\|\sfH_k\|\|\BFsn_k\|\le(\khess)\|\BFsn_k\|\le \frac{\khess}{\ksn}\|\BFc_k\|,$$
    but we also know that 
    \[\|\BFGN_k\|\ge \|(\sfAN_k)^\top\BFG_k\|-\|(\sfAN_k)^\top\BFG_k-\BFGN_k\|\ge \left(\mu-\frac{\mu}{2(1+\frac{\khess}{\ksn})}-\frac{\khess\mu}{2\ksn(1+\frac{\khess}{\ksn})}\right)\Delta_k=\frac{\mu}{2}\Delta_k.\] Therefore, using Lemma~\ref{lem: model improvment bound by tangent} and Assumption~\ref{as: tangent cauchy}, we get that 
    \[\Psikp-\Psik\le-\eta\frac{\mu}{2\sigmax}\kmt\ktanc\min\left(\at,\frac{\mu}{2\khess}\right)\Delta_k^2:=-\kpsirg\Delta_k^2.\] We hence write that $\Psikp-\Psik\le-\min\left(\kpsirc\Delta_k,\kpsirg\Delta_k^2\right)$ for all $k\in\mcS$.

    However, we also know from Theorem~\ref{th: est error bound} that, given $\kerror>0$ for $k\notin \mcS$ large enough, $$\Psikp-\Psik\le 2\kerror\sigmin^{-1}\min\left(\Delta_k,\Delta_k^2\right).$$ Let $K$ be the random large iteration number such that for all $k\ge K$ the former holds. In fact, between two large enough consecutive successful iterations, $k_1$ and $k_2$, we can write $$
    \sum_{k=k_1+1}^{k_2-1}\Psikp-\Psik\le 2\kerror\sigmin^{-1}\min\left(\Delta_{k_1},\Delta_{k_1}^2\right).$$ If we choose $\kerror<\frac{\sigmin}{2}\min(\kpsirc,\kpsirg)$ we get
    \begin{align*}
     -\infty&<\Psi_{\min}-\Psi_0\le \sum_{k}\Psikp-\Psik \\
     &= \sum_{\substack{k< K
                  }}\Psikp-\Psik + \sum_{\substack{k\ge K
                  }}\Psikp-\Psik   \\
     & \le \sum_{\substack{k< K
     \\
                  k\in\mcS
                  }}
                  \Psikp-\Psik +\sum_{\substack{k\ge K\\
                  k\in\mcS}} \underbrace{\left(-\min(\kpsirc,\kpsirg)+\frac{2\kerror}{\sigmin}\right)}_{<0}\min(\Delta_k,\Delta_k^2),
    \end{align*} completing the proof. 
\end{proof}

The next lemma makes an important connection between $\Delta_k$ and both first-order optimality criteria.
\begin{lemma}\label{lem:delta and optimality}
    If Assumptions~\ref{as: bounded true Hess}, \ref{as: bounded lag poly}, \ref{as: bounded hess}, \ref{as: ortho AN}, \ref{as: normal enough}, \ref{as: normal cauchy}, \ref{as: tangent cauchy}, and \ref{as: A bounded} hold, then $$\mbP\left(\left(\Delta_k<\kdelta\pi_k\right)\cap\left(k\notin\mcS\right)\text{ i.o.}\right)=0$$ for 
    \[\kdelta=\min\left(\frac{1}{\mu},\frac{1}{\at\khess},\frac{(1-\eta)\at\ktanc}{\kphi},\frac{(1-\eta)\nu\sigmin\knormal\amin\an}{\kphi(\kuc+\kug+\kmg\Deltamax)},\frac{1}{\frac{\kmg}{\Deltamax}+\kmg+\frac{\an(1+\amax^2)}{\amin}}\right).\]
    where $\kphi=\kerr(1+\sigmax)$, and $\kerr$ is the constant defined in Lemma~\ref{lem: model error}.
\end{lemma}
\begin{proof}
Let $\omega\in\Omega_1$ where $\Omega_1$ is the largest set of probability 1. We show that if $\Delta_k(\omega)\le \kdelta\pi_k(\omega)$, then iteration $k$ will eventually have to be successful. For ease of exposition, we drop $\omega$ henceforth. 

Note, $\kdelta \le \mu^{-1}$ implies that $\mu\Delta\le \pi_k$; in other words, the second success criterion holds. To prove that when $k$ is large enough, it must be successful, we only need to show that eventually $\rho_k\ge \eta$. From Lemma~\ref{lem: model error} we know for large enough $k$, $|\Phi_k^s-\Ms|\le \kphi\Delta_k^2$ almost surely.

If $\|\BFc_k\|=0$, then $$\|\BFGN_k\|=\|(\sfAN_k)^\top \BFG_k\|$$ and $$\Mz-\Ms=\Mtz-\Mtt \ge \ktanc \|\BFGN_k\|\min\left(\at\Delta_k,\frac{\|\BFGN_k\|}{\khess}\right),$$ but $$\pi_k=\|(\sfAN_k)^\top \BFG_k\|=\|\BFGN_k\|\ge \kdelta^{-1}\Delta_k.$$ So, $$\min\left(\at\Delta_k,\frac{\|\BFGN_k\|}{\khess}\right)\ge \Delta_k\min\left(\at,\frac{1}{\kdelta\khess}\right)=\at\Delta_k$$ since $\kdelta\le (\at\khess)^{-1}$. Therefore, $$|1-\rho_k|\le\frac{\kphi\Delta_k^2}{\ktanc\at\kdelta^{-1}\Delta_k^2}\le 1-\eta$$ since $$\kdelta\le \frac{(1-\eta)\ktanc\at}{\kphi}.$$

If $\|\BFc_k\|>0$, first note, $$\kdelta^{-1}\Delta_k\le \|(\sfAN_k)^\top\BFG_k\|+\|\BFc_k\|\le \|\BFG_k-\nabla f(\BFX_k)\|+\kuc+\kug.$$ Hence $\Delta_k\le \kdelta(\kuc+\kug+\kmg\Deltamax)$. Now we look at two cases. First,  $$\|\BFc_k\|>\frac{1+\amax^2}{\amin}\an\Delta_k.$$ Here, we get, $$|1-\rho_k|\le\frac{\kphi\kdelta(\kuc+\kug+\kmg\Deltamax)}{\nu\sigmin\knormc\amin\an}\le  1-\eta.$$  
Alternatively, if $$\|\BFc_k\|\le \frac{1+\amax^2}{\amin}\an\Delta_k,$$ 
since we know $$\|\BFc_k\|\ge \frac{1}{\kdelta}\Delta_k-\|(\sfAN_k)^\top\BFG_k\|\ge (\frac{1}{\kdelta}-\kmg)\Delta_k-\kug,$$  we reach a contradiction with the definition of $\kdelta$.
\end{proof}
The consequence of Lemma~\ref{lem:delta and optimality} is that in each sequence generated by Algorithm~\ref{alg: sqp} with a nonzero probability, for large enough $k$, $\Delta_k\ge \trdec\kdelta\pi_k$ since as soon as the radius becomes small enough, the next iteration will become successful, expanding the radius. We now have all the ingredients to show the convergence of SQP-ASTRO-DF to first-order optimality almost surely.

\begin{theorem}
    If Assumptions~\ref{as: bounded true Hess}
    --\ref{as: A bounded} hold, then the iterates generated by Algorithm~\ref{alg: sqp} satisfy
    \begin{itemize}
        \item[(i)] $\mbP\left(\|\BFc_k\|\to 0\right)=1$;
        \item[(ii)] $\mbP\left(\|(\sfAN_k)^\top\nabla f(\BFX_k)\|\to 0\right)=1$.
    \end{itemize}
\end{theorem}
\begin{proof}
    (i) is proven trivially given that from Lemma~\ref{lem: delta converges} $\Delta_k\to 0$ almost surely, and from Lemma~\ref{lem:delta and optimality}, we eventually have $$\Delta_k\ge \kdelta (\|\BFc_k\|+\|(\sfAN_k)^\top \BFG_k\|)$$ almost surely. 

    Two prove (ii), we also note from Theorem~\ref{th: gradient error bound} that for sufficiently large $k$ almost surely
    \[\|(\sfAN_k)^\top\nabla f(\BFX_k)\|\le \|(\sfAN_k)^\top \BFG_k\| + \|\nabla f(\BFX_k)-\BFG_k\|\le \|(\sfAN_k)^\top \BFG_k\| + \kmg\Delta_k,\] which completes the proof given that both terms vanish almost surely again by Lemma~\ref{lem: delta converges} and Lemma~\ref{lem:delta and optimality}. 
\end{proof}

\section{Algorithm Implementation and Numerical Experiment}
To test the performance of the proposed Algorithm~\ref{alg: sqp}, we perform a validation experiment on the SAN problem, described in Section~\ref{sec:intro} with $d=13, b=5$, and $e_j = 1$ for $j=1, \dots 13$, using an implementation of SQP-ASTRO-DF compatible with the SimOpt library \citep{eckhensha21,Eckman_SimOpt}. We run 10 macroreplications of Algorithm~\ref{alg: sqp} with parameters $\eta = 0.2,\: \trdec = 0.5, \:\trinc = 2.5, \: \an = 0.9, \: \at = 0.1, \: \nu = 0.01, \: \tau_1 = 3, \: \tau_2 = 2, \:\mu = 0.1, \: \sigbmax = 1 \times 10^8,$ and $ \Deltamax = 100$. We do not claim that these algorithm parameters are optimal, and future hyper-parameter tuning is likely warranted. We also set a feasibility tolerance of $0.01$, so that when $\|\BFc_k\| \leq 0.01$, we set the normal step $\BFsn_k = 0$ and allow the tangent step to use the entire trust-region radius with $\|\BFSt_k\| \leq\Delta_k$. Such a  less restrictive tolerance (than say $1 \times 10^{-4}$ or $1 \times 10^{-8}$) is reasonable for this problem given that in an operations engineering problem a cost constraint violation less than $\$0.01$ is negligible. 

Fig.~\ref{fig: results} shows the objective and constraint violation progress of macroreplications of SQP-ASTRO-DF over a budget of $20,000$ simulation runs, where constraint violation is measured by $\| \BFc (\BFx)\|$. In Fig.~\ref{fig: obj progress}, each macroprelication turns red once the feasibility tolerance has been satisfied, meaning that those iterates are ``feasible''. We observe that SQP-ASTRO-DF successfully 
improves the objective value (project completion time) and achieves terminal solutions satisfying the designated feasibility tolerance (cost constraint) on all macroreplications. Notably we do not compare SQP-ASTRO-DF against other DFO constrained algorithms here as we are not attempting to claim superior performance of our algorithm against others, which is a planned extension of this work. 

\begin{figure}
  \centering
  \subfloat[Objective Progress]{%
    \includegraphics[width=0.45\textwidth]{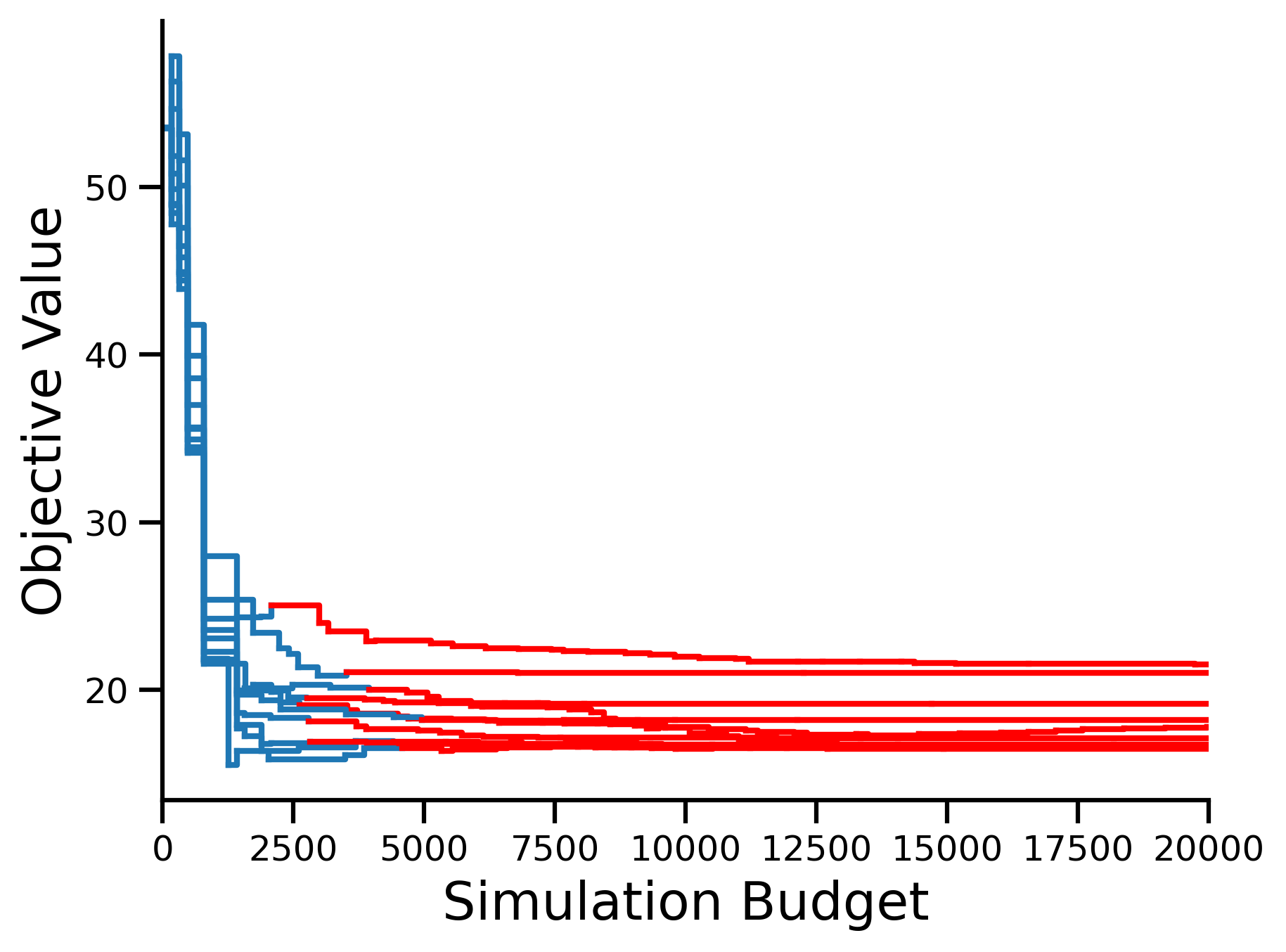}
    \label{fig: obj progress} 
  }
  \hspace{0.05\textwidth} 
  \subfloat[Constraint Violation Progress]{%
    \includegraphics[width=0.45\textwidth]{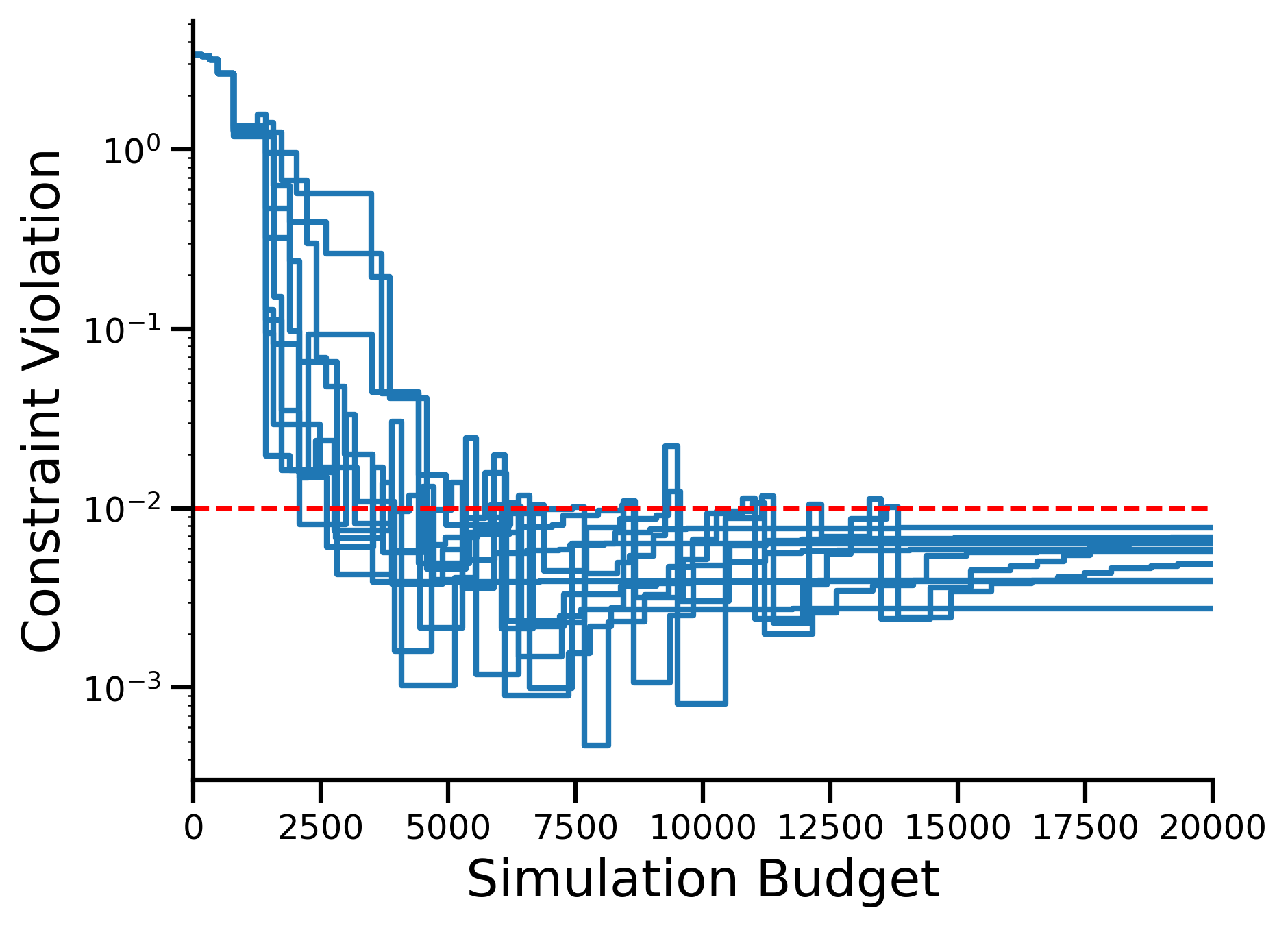}
    \label{fig: scatter} 
  }
  \caption{Performance of SQP-ASTRO-DF on SAN}\label{fig: results}
\end{figure}

\section{Conclusion}
This paper introduces SQP-ASTRO-DF, an algorithm for simulation optimization with deterministic nonlinear equality constraints. The algorithm adopts an adaptive sampling scheme that uses the trust-region radius as a proxy for stationarity, which must jointly reflect both feasibility and optimality. Building on the classical Byrd--Omojokun composite-step framework with minimal modification, we introduce a criticality measure $\pi_k$ that naturally handles this joint requirement. We establish almost-sure convergence of the algorithm and provide initial experimental validation showing convergence to feasible solutions. Directions for future work include comparative analysis against other algorithms and implementation improvements, extension to nonlinear inequality constraints, and derivative-free and stochastically constrained settings.

\section*{ACKNOWLEDGMENTS}
This research was partially supported by National Science Foundation Grant CMMI-2226347, OAC-2410949, Office of Naval Research Grant N000142412398, and Australian Research Council award DE240100006.
\footnotesize

\bibliographystyle{plain}
\bibliography{refs}
\vspace{0.5in}

\end{document}